\documentclass[11pt]{article}

\usepackage[ascii]{inputenc}
\usepackage{amsfonts}
\usepackage{amsmath,amssymb,amsthm,amsfonts,enumerate}
\usepackage{graphicx}
\usepackage{float}
\usepackage[scriptsize,nooneline]{subfigure}
\usepackage{indentfirst}
\usepackage{cite}
\usepackage{color}
\usepackage{mathrsfs}
\usepackage{epstopdf}
\usepackage{setspace}
\usepackage[labelfont=footnotesize,font=footnotesize]{caption}
\usepackage[colorlinks, citecolor=red]{hyperref}

\newtheorem{theorem}{Theorem}[section] 
\newtheorem{corollary}{Corollary}[section] 
\newtheorem{lemma}{Lemma}[section]

\newtheorem{definition}{Definition}[section] 
\newtheorem{example}{Example}[section] 

\newtheorem{remark}{Remark}[section] 

\def\geq{\geqslant}\def\leq{\leqslant}

\def\lim{\limits}
\begin{document}
\title{\bf   Hyponormal block Toeplitz operators with non-harmonic symbols on the weighted Bergman space }
\author{Guangyang Fu, Jiang Zhou\thanks{Corresponding author. The research was supported by National Natural Science Foundation of China (12061069).} \\[.5cm]}

\date{}
\maketitle
{\bf Abstract:}\quad{
\hyphenpenalty=5000
\tolerance=1000
In this paper, we discuss hyponormal block Toeplitz operators $T_{\Phi}$ over the vector-valued weighted Bergman space $A_\alpha^2\left(\mathbb{C}^n\right)$. And two conditions about hyponormal block Toeplitz operators $T_{\Phi}$ on $A_\alpha^2\left(\mathbb{C}^n\right)$ were discussed separately, where $ \Phi(z)=A z^p \bar{z}^q + B z^s \bar{z}^t $, $A,B$ are any n-order complex square matrices. }\par


\baselineskip 15pt

\section{Introduction}
In 1970, Halmos \cite{HA1970}, in his research on Hilbert spaces, proposed: Is every subnormal Toeplitz operator $T_{\varphi}$ either analytic or normal? In 1975, Amemiya \cite{AI1975} obtained that Halmos' conjecture holds for quasinormal Toeplitz operators $T_{\varphi}$. In 1976, Abrahamse \cite{AM1976} showed that if $\varphi$ is a function of bounded type, Halmos' conjecture is true. It was not until 1984 that Cowen and Long \cite{CL1984} responded negatively to Halmos' conjecture. 

Since hyponormal Toeplitz operators $T_{\varphi}$ play an important role in solving Halmos's conjecture,
In 1988, Cowen \cite{CC1988} provided a characterization for the hyponormality of $T_{\varphi}$ on the Hardy space utilizing the dilation theorem of Sarason \cite{SD1967}. Subsequently, much work has been conducted by connecting the hyponormality of $T_{\varphi}$ to function space theory and interpolation problems, for example, \cite{HI1999,Zhu1995,NTTK}.

In 2006, Gu \cite{GC2006} characterized block hyponormal Toeplitz operators $T_{\Phi}$ on the vector-valued Hardy space utilizing Nagy-Foias's general commutant lifting theorem \cite{FF1968,FF1993}. Additional relevant research about $T_{\Phi}$ includes \cite{CR2012,HI2011}.

We would discuss hyponormal block Toeplitz operators $T_{\Phi}$ on the vector-valued weighted Bergman space, where $ \Phi(z)=A z^p \bar{z}^q + B z^s \bar{z}^t $, $A,B$ are any n-order complex square matrices. In 2019, Fleeman and Liaw \cite{FM2019} presented sufficient conditions for the hyponormality of $T_{az^m \bar z^n+bz^s \bar z^t}$ on the Bergman space. In 2021, Le and Simanek \cite{BS2021} gave an equivalence condition for the hyponormality of $T_{z^n+C|z|^s}$ on the weighted Bergman space, and then Kim and Lee \cite{KL2021,KL2023} further considered two conditions for the hyponormality of $T_{az^m \bar z^n+bz^s \bar z^t}$ on the Bergman space and extended the outcomes to the weighted Bergman space. For Bergman spaces, in fact, the lack of a commutant lifting theorem makes it hard to determine hyponormal (block) Toeplitz operators.

In 2019, Lee \cite{LJ2019} considered hyponormal block Toeplitz operators $T_{\Phi}$ on the vector-valued weighted Bergman space, where $ \Phi$ are matrix-valued polynomial symbols with circulants as coefficients. We will also consider the hyponormality of $ T_{\Phi}$, where $ \Phi(z)=A z^p \bar{z}^q + B z^s \bar{z}^t $, $A,B$ are circulants.

Below is the structure of our paper: in Section 2, we will provide the fundamental definitions and lemmas. In Section 3, we will give two conditions for block Toeplitz operators $ T_{\Phi}$ to be hyponormal on the vector-valued weighted Bergman space, where $ \Phi(z) = A z^p \bar{z}^q + B z^s \bar{z}^t $, $A,B $ are any n-order complex square matrices. In Section 4, we will consider two conditions for the hyponormality of $ T_{\Phi}$, where matrices $A,B$ in $\Phi$ are circulants.

Throughout the paper, the following notations are used: Let $ \mathbb{D}$ be the open unit disk. Write $M_{n}$ be the set of all $n \times n$ complex matrices. Let $M_n^*$ denote the set of matrices obtained by taking the conjugate of each element of $M_n$. Let $\mathbb{O}$ denote the null matrix and $I_n$ denote the identity matrix. Let $L^{\infty}\left(M_n\right)$ denote the space of $M_n$-valued essentially bounded Lebesgue measurable functions on $ \mathbb{D} $. Let $H(\mathbb{D})$ denote the space of analytic functions in $ \mathbb{D} $. For $A \in M_n$, let $A \geq 0$ denote that $A$ is a Hermite positive semidefinite matrix. 

\section{\hspace{-0.6cm}{\bf }~~Some Preliminaries and Notations\label{s2}}
  \begin{definition} \cite{HA1970}
	For a bounded linear operator $\mathcal{T}$ on a Hilbert space, $\mathcal{T}$ is hyponormal if its self-commutator $\left[\mathcal{T}^*, \mathcal{T}\right]:=\mathcal{T}^* \mathcal{T}-\mathcal{T} \mathcal{T}^* \geq 0$, where operator $\mathcal{T}^*$ denotes the adjoint of operator $\mathcal{T}$. 
\end{definition} 

  Theories of Bergman spaces can go back to Bergman's work \cite{BS1950}.
\begin{definition} \cite{BS1950}
	Given $-1<\alpha<\infty$, the weighted Bergman space $A_\alpha^2(\mathbb{D})$ is defined by
	$$ 
	A_\alpha^2(\mathbb{D})= H(\mathbb{D}) \cap L^{ 2 }(\mathbb{D}, dA_{ \alpha }),
	$$
	where $d A_\alpha(z)=(\alpha+1)\left(1-|z|^2\right)^\alpha d A(z)$, $ dA$ is the Euclidean area measure on the complex plane.
\end{definition}


Let $L_\alpha^2\left(\mathbb{C}^n\right)$ is a Hilbert space of $\mathbb{C}^n$-valued Lebesgue square integrable functions with the inner product
$$
\langle f, g\rangle:=\int_{\mathbb{D}} \langle f(z), g(z)\rangle_{\mathbb{C}^n} d A_\alpha(z), \quad f, g \in L_\alpha^2\left(\mathbb{C}^n\right),
$$
where $\langle,\rangle_{\mathbb{C}^n} \text { is the standard inner product on } \mathbb{C}^n \text {. }$

\begin{definition} \cite{BS1950}
	Given $-1<\alpha<\infty$, the vector-valued weighted Bergman space
	$A_\alpha^2\left(\mathbb{C}^n\right)$ is defined by
	$$ 
	A_\alpha^2\left(\mathbb{C}^n\right)= H(\mathbb{D}^n) \cap L_\alpha^2\left(\mathbb{C}^n\right).
	$$
\end{definition}


Theories of Toeplitz operators can go back to Toeplitz, Halmos, and Brown's work \cite{HB1964,TO1911}. For further details on Bergman spaces and Toeplitz operators, see \cite{Zhu2007,HZ2000}.
\begin{definition} \cite{HB1964,TO1911}
	Given the matrix-valued function $\Phi \in L^{\infty}\left(M_n\right)$, the block Toeplitz operator $T_{\Phi}$ with symbol $\Phi$ on the vector-valued weighted Bergman space $A_\alpha^2\left(\mathbb{C}^n\right)$ is defined as follows:
	$$
	T_{\Phi} f:=P(\Phi f), \quad  f \in A_\alpha^2\left(\mathbb{C}^n\right),
	$$
	where the orthogonal projection that maps $L_\alpha^2\left(\mathbb{C}^n\right)$ onto $A_\alpha^2\left(\mathbb{C}^n\right)$ is represented by $P$. 
\end{definition}
\begin{remark}
	It is easy to verify that $T_{\Phi_1+\Phi_2} =T_{\Phi_1}+T_{\Phi_2}$, $T_{\Phi_1}^* =T_{\Phi_1^*}$, where $\Phi_1 , \Phi_2 \in L^{\infty}\left(M_n\right) $.
\end{remark}

We will utilize the following notations for our convenience.\\
Given $c_i \in \mathbb{C}~(i = 1,2 \dots n)$, the circulant is a matrix, which has the following form:
$$
\operatorname{cir}\left[c_1, c_2, \cdots, c_{n}\right] :=
\left(\begin{array}{ccccc}
	c_1 & c_2 & \cdots & \cdots & c_{n} \\
	c_{n} & c_1 & \ddots & & \vdots \\
	\vdots & \ddots & \ddots & \ddots & \vdots \\
	c_2 & \cdots & \cdots & c_{n} & c_1
\end{array}\right).
$$
Given $j \in \mathbb{N}_+\cup {\left\{0\right\}}$, $c_{i}^{(l)} \in \mathbb{C}$ for $l=1,2, \cdots, n$, write
$$
K_j(z)=\sum_{i=0}^{\infty} z^{j+i} \left(
c_{i}^{(1)}, c_{i}^{(2)}, \dots, c_{i}^{(n)}
\right)^{T}.
$$
Given $p,q \in \mathbb{N}_+\cup {\left\{0\right\}}$, write
$$
\Lambda_\alpha(p)=\frac{\Gamma(p+1) \Gamma(\alpha+2)}{\Gamma(p+\alpha+2)}, \quad \Lambda_\alpha(p, q)=\frac{\Gamma(p+1)^2 \Gamma(p-q+\alpha+2) \Gamma(\alpha+2)}{\Gamma(p+\alpha+2)^2 \Gamma(p-q+1)}.
$$

Next, we review the properties of projections and norms on $A_\alpha^2(\mathbb{D})$. 

\begin{lemma} \cite{BS1950}
	Given $p, q \in \mathbb{N}_+\cup {\left\{0\right\}}$, then
	\begin{align*}
		\left\langle z^q, z^p\right\rangle= \begin{cases}\Lambda_\alpha(p) & \text { if } p=q \\ 0 & \text { if } p\neq q \end{cases}\quad
		P\left(\bar{z}^q z^p\right)= \begin{cases} \displaystyle\frac{\Lambda_\alpha(p, q)}{\Lambda_\alpha(p)} z^{p-q}& \text { if } p \geq q \\ 0 & \text { if } p < q\end{cases}.
	\end{align*}
\end{lemma}

$\Lambda_\alpha(p)$ and $\Lambda_\alpha(p,q)$ have the following relation:

\begin{lemma} \cite{KL2023}
	Let $p, q, s, t \in \mathbb{N}_+\cup {\left\{0\right\}}$.
	
	(i) If $i \geq p-q$, then $\Lambda_\alpha(p+i, q) \geq \Lambda_\alpha(q+i, p)$. 
	
	(ii) If $i \geq p-q=s-t>0$, then
	$$
	\frac{\Lambda_\alpha(p+i) \Lambda_\alpha(s+i)}{\Lambda_\alpha(p-q+i)} \geq \frac{\Lambda_\alpha(q+i) \Lambda_\alpha(t+i)}{\Lambda_\alpha(q-p+i)} .
	$$
	
	(iii) Given $0 \leq i<p-q=t-s$, if $t \geq p$, then $\displaystyle\frac{\Lambda_\alpha(t+i, s)}{\Lambda_\alpha(p+i, q)}$ is increasing in $i$, and $t<p$, then $\displaystyle\frac{\Lambda_\alpha(t+i, s)}{\Lambda_\alpha(p+i, q)}$ is decreasing in $i$.
\end{lemma}

\section{\hspace{-0.6cm}{\bf }~~Conditions for hyponormal block Toeplitz opeartors\label{s3}}

In this section, given $ p,q,s,t \in \mathbb{N}_+\cup {\left\{0\right\}}$, we consider two conditions for the hyponormality of $T_{\Phi}$ on $A_\alpha^2\left(\mathbb{C}^n\right)$, where $ \Phi(z)=A z^p \bar{z}^q + B z^s \bar{z}^t $, $ A,B \in M_n$. 

The following lemma on $A_\alpha^2\left(\mathbb{C}^n\right)$ from a direct calculation by Lemma 2.1.
\begin{lemma}
	Given $p,q \in \mathbb{N}_+\cup {\left\{0\right\}}$, $a_{ij} \in \mathbb{C}$ and $A = (a_{ij})_{n \times n}\in M_n$, then
	\begin{align*}
		&(i) \left\|A \bar{z}^q K_p(z)\right\|^2=\sum_{i=0}^{\infty} \Lambda_\alpha(p+i+q) \sum_{k=1}^n \left| \sum_{l=1}^n a_{k l} c_i^{(l)}\right|^2.\\
		&(ii) \left\|P \left(A \bar{z}^q K_p(z)\right)\right\|^2= \begin{cases}\sum_{i=0}^{\infty} \Lambda_\alpha(p+i,q) \sum_{k=1}^n \left| \sum_{l=1}^n a_{k l} c_i^{(l)}\right|^2 & \text { if } p \geq q \\ \sum_{i=q-p}^{\infty} \Lambda_\alpha(p+i,q) \sum_{k=1}^n \left| \sum_{l=1}^n a_{k l} c_i^{(l)}\right|^2 & \text { if } p<q \end{cases}.
	\end{align*}
\end{lemma}

First, given $ p,q \in \mathbb{N}_+\cup {\left\{0\right\}}$, we consider the case when $\Phi(z)=A z^p \bar{z}^q $, $ A = (a_{ij})_{n \times n} \in M_n $.
\begin{theorem}
	Given $p,q \in \mathbb{N}_+\cup {\left\{0\right\}}$, let $\Phi(z)=A z^p \bar{z}^q$, $\mathbb{O} \neq A = (a_{ij})_{n \times n} \in M_n $. If for all $k = 1, 2, \dots, n$, 
	\begin{equation}
		\left(\begin{matrix}
			a_{k1}\\
			a_{k2}\\
			\dots\\
			a_{kn}
		\end{matrix}\right)(\bar{a}_{k1},\bar{a}_{k2},\dots,\bar{a}_{kn}) - \left(\begin{matrix}
			\bar{a}_{k1}\\
			\bar{a}_{k2}\\
			\dots\\
			\bar{a}_{kn}
		\end{matrix}\right)(a_{k1},a_{k2},\dots,a_{kn}) \geq 0,
	\end{equation}
	then $T_{\Phi}$ is hyponormal if and only if $p \geq q$.
\end{theorem}
\begin{proof}
	We only have to prove:
	\begin{equation}
		\left\langle\left(T_{\Phi}^* T_{\Phi}-T_{\Phi} T_{\Phi}^*\right) K_0(z), K_0(z) \right\rangle \geq 0.
	\end{equation}
	If $p \geq q$, by Lemma 3.1, $T_{\Phi}$ is hyponormal, which is equivalent to
	\begin{equation}
		\begin{aligned}
			&\left\| P\left(A \bar{z}^q K_p(z)\right) \right\|^2- \left\| P\left(A^* \bar{z}^p K_q(z)\right) \right\|^2\\
			=&\sum_{i=0}^{\infty} \Lambda_\alpha(p+i,q) \sum_{k=1}^n \left| \sum_{l=1}^n a_{k l} c_i^{(l)}\right|^2 - \sum_{i=p-q}^{\infty} \Lambda_\alpha(q+i,p) \sum_{k=1}^n \left| \sum_{l=1}^n \bar{a}_{k l} c_i^{(l)}\right|^2 \geq 0.
		\end{aligned}
	\end{equation}
	Since (1) holds for $k = 1, 2, \dots, n$, then $\left| \sum_{l=1}^n a_{k l} c_i^{(l)}\right|^2 - \left| \sum_{l=1}^n \bar{a}_{k l} c_i^{(l)}\right|^2 \geq 0$ for $k = 1, 2, \dots, n$. Using Lemma 2.2 (i), then (3) holds, hence $T_{\Phi}$ is hyponormal.
	
	Similarly, if $p < q$ and (1) holds for $k = 1, 2, \dots, n$, then $T_{\Phi^*}$ is hyponormal. We conclude from $ T_{\Phi}^* T_{\Phi}-T_{\Phi} T_{\Phi}^* = -\left(T_{\Phi^*}^* T_{\Phi^*}-T_{\Phi^*} T_{\Phi^*}^*\right)$ that $T_{\Phi}$ is never hyponormal, which completes the proof.
\end{proof}
\begin{remark}
	If $A = \mathbb{O}$, then $\Phi \equiv \mathbb{O}$, hence $T_{\Phi}$ is always hyponormal, where $\Phi(z)=A z^p \bar{z}^q $.
\end{remark}
\begin{corollary}
	Given $p,q \in \mathbb{N}_+\cup {\left\{0\right\}}$, if $\Phi(z)=A z^p \bar{z}^q$, $ \mathbb{O} \neq A = (a_{ij})_{n \times n} \in M_n $ satisfying $A^*=A$ (i.e. A is a real matrix), then $T_{\Phi}$ is hyponormal if and only if $p \geq q$.
\end{corollary}
\begin{proof}
	Immediate from Theorem 3.1.
\end{proof}
\begin{corollary}
	Given $p,q \in \mathbb{N}_+\cup {\left\{0\right\}}$, $p > q$, if $\Phi(z)=a I_n z^p \bar{z}^q+ b I_n z^q \bar{z}^p $, $a,b \in \mathbb{C}$, then $T_{\Phi}$ is hyponormal if and only if $a \geq b$.
\end{corollary}
\begin{proof}
	By Lemma 3.1, then $T_{\Phi}$ is hyponormal, which is equivalent to
	\begin{align*}
		& \left|a\right|^2 \sum_{i=0}^{\infty} \Lambda_\alpha(p+i,q) \sum_{l=1}^n \left|  c_i^{(l)}\right|^2 -   \left|a\right|^2 \sum_{i=p-q}^{\infty} \Lambda_\alpha(q+i,p) \sum_{l=1}^n \left|  c_i^{(l)}\right|^2 \\
		& +\left|b\right|^2 \sum_{i=p-q}^{\infty} \Lambda_\alpha(p+i,q) \sum_{l=1}^n \left|  c_i^{(l)}\right|^2 -  \left|b\right|^2 \sum_{i=0}^{\infty} \Lambda_\alpha(q+i,p) \sum_{l=1}^n \left|  c_i^{(l)}\right|^2 \\
		=&\left( \left|a\right|^2 - \left|b\right|^2 \right) \left\{ \sum_{i=0}^{\infty} \Lambda_\alpha(p+i,q) \sum_{l=1}^n \left|  c_i^{(l)}\right|^2 - \sum_{i=p-q}^{\infty} \Lambda_\alpha(q+i,p) \sum_{l=1}^n \left|  c_i^{(l)}\right|^2 \right\} \geq 0.
	\end{align*}
	Applying Lemma 2.2 (i), we obtain the conclusion.
\end{proof}

Next, given $ p,q,s,t \in \mathbb{N}_+\cup {\left\{0\right\}}$, we consider the case when $ \Phi = A z^p \bar{z}^q + B z^s \bar{z}^t$, $ A,B \in M_n $.
\begin{theorem}
	Given $p,q,s,t \in \mathbb{N}_+\cup {\left\{0\right\}}$, $p-q=s-t > 0$, $\Phi(z)=A z^p \bar{z}^q+ B z^s \bar{z}^t$, $ A = (a_{ij})_{n \times n} \in M_n $, $B = (b_{ij})_{n \times n} \in M_n $, $\operatorname{Re} \left( a_{i j} \bar{b}_{i j} \right) \geq 0$ for all $i,j = 1, 2, \dots, n$ and for all $k = 1, 2, \dots, n$,
	\begin{equation}
		\left(\begin{matrix}
			a_{k1}\\
			a_{k2}\\
			\dots\\
			a_{kn}
		\end{matrix}\right)(\bar{a}_{k1},\bar{a}_{k2},\dots,\bar{a}_{kn}) - \left(\begin{matrix}
			\bar{a}_{k1}\\
			\bar{a}_{k2}\\
			\dots\\
			\bar{a}_{kn}
		\end{matrix}\right)(a_{k1},a_{k2},\dots,a_{kn}) \geq 0,
	\end{equation}
	\begin{equation}
		\left(\begin{matrix}
			b_{k1}\\
			b_{k2}\\
			\dots\\
			b_{kn}
		\end{matrix}\right)(\bar{b}_{k1},\bar{b}_{k2},\dots,\bar{b}_{kn}) - \left(\begin{matrix}
			\bar{b}_{k1}\\
			\bar{b}_{k2}\\
			\dots\\
			\bar{b}_{kn}
		\end{matrix}\right)(b_{k1},b_{k2},\dots,b_{kn}) \geq 0,
	\end{equation}
	and for all $k = 1, 2, \dots, n$,
	\begin{equation}
		\left(\begin{matrix}
			a_{k1}\\
			a_{k2}\\
			\dots\\
			a_{kn}
		\end{matrix}\right)( \bar{b}_{k1}, \bar{b}_{k2},\dots, \bar{b}_{kn}),
	\end{equation}
	are diagonal matrices, then $T_{\Phi}$ is hyponormal.
\end{theorem}
\begin{proof}
	By Lemma 3.1, $T_{\Phi}$ is hyponormal, which is equivalent to
	\begin{equation}
		\begin{aligned}
			&\left\| P\left(A \bar{z}^q K_p(z)\right) + P\left(B \bar{z}^t K_s(z)\right) \right\|^2- \left\| P\left(A^* \bar{z}^p K_q(z)\right) + P\left(A^* \bar{z}^s K_t(z)\right)  \right\|^2\\
			=&\sum_{i=0}^{\infty} \Lambda_\alpha(p+i,q) \sum_{k=1}^n \left| \sum_{l=1}^n a_{k l} c_i^{(l)}\right|^2 - \sum_{i=p-q}^{\infty} \Lambda_\alpha(q+i,p) \sum_{k=1}^n \left| \sum_{l=1}^n \bar{a}_{k l} c_i^{(l)}\right|^2 \\
			& +\sum_{i=0}^{\infty} \Lambda_\alpha(s+i,t) \sum_{k=1}^n \left| \sum_{l=1}^n b_{k l} c_i^{(l)}\right|^2 - \sum_{i=s-t}^{\infty} \Lambda_\alpha(t+i,s) \sum_{k=1}^n \left| \sum_{l=1}^n \bar{b}_{k l} c_i^{(l)}\right|^2 \\
			& +2 \sum_{i=0}^{\infty} \sum_{k=1}^n \operatorname{Re}  \left( \left( \sum_{l=1}^n a_{k l} c_i^{(l)} \right) \left( \sum_{l=1}^n \bar{b}_{k l} \bar{c}_{i}^{(l)} \right) \right) \frac{\Lambda_\alpha(p+i) \Lambda_\alpha(s+i)}{\Lambda_\alpha(p-q+i)} \\
			& - 2 \sum_{i=p-q}^{\infty}  \sum_{k=1}^n \operatorname{Re} \left( \left( \sum_{l=1}^n \bar{a}_{k l} c_i^{(l)} \right) \left( \sum_{l=1}^n b_{k l} \bar{c}_i^{(l)} \right) \right) \frac{\Lambda_\alpha(q+i) \Lambda_\alpha(t+i)}{\Lambda_\alpha(q-p+i)} \geq 0.
		\end{aligned}
	\end{equation}
	Since both of (4) and (5) hold for $k = 1, 2, \dots, n$, then for $k = 1, 2, \dots, n$,
	\begin{equation}
		\left| \sum_{l=1}^n a_{k l} c_i^{(l)}\right|^2 - \left| \sum_{l=1}^n \bar{a}_{k l} c_i^{(l)}\right|^2 \geq 0.
	\end{equation}
	\begin{equation}
		\left| \sum_{l=1}^n b_{k l} c_i^{(l)}\right|^2 - \left| \sum_{l=1}^n \bar{b}_{k l} c_i^{(l)}\right|^2 \geq 0.
	\end{equation}
	Since $\operatorname{Re} \left( a_{i j} \bar{b}_{i j} \right) \geq 0$ for all $i,j = 1, 2, \dots, n$ and (6) are diagonal matrices for all $k = 1, 2, \dots, n$, then for all $k = 1,2, \dots, n$,
	\begin{equation}
		\begin{aligned}
			& \quad \operatorname{Re} \left( \left(\sum_{l=1}^n a_{k l} c_i^{(l)} \right) \left( \sum_{l=1}^n \bar{b}_{k l} \bar{c}_i^{(l)} \right) \right) \\
			&= \sum_{l=1}^n \operatorname{Re} \left( a_{k l} \bar{b}_{k l} \right) \left|c_{i}^{(l)}\right|^2 + \sum_{k_1,k_2=1, k_1 \neq k_2}^n \operatorname{Re} \left( a_{k k_1} \bar{b}_{k k_2} c_i^{(k_1)} \bar{c}_i^{(k_2)} \right) \\
			&=\sum_{l=1}^n \operatorname{Re} \left( a_{k l} \bar{b}_{k l} \right) \left|c_{i}^{(l)}\right|^2 \geq 0.
		\end{aligned}
	\end{equation}
	Combined with (8), (9), (10) and Lemma 2.2 (i) (ii), then (7) holds, hence $T_{\Phi}$ is hyponormal.
\end{proof}
\begin{corollary}
	Given $p,q,s,t \in \mathbb{N}_+\cup {\left\{0\right\}}$, $p-q=s-t > 0$, if $\Phi(z)=A z^p \bar{z}^q+ B z^s \bar{z}^t$, $ A = (a_{ij})_{n \times n} \in M_n $, $B = (b_{ij})_{n \times n} \in M_n $ satisfying $A^*=A$ and $B^*=B$ and for all $k = 1, 2, \dots, n$, 
	\begin{equation}
		\left(\begin{matrix}
			a_{k1}\\
			a_{k2}\\
			\dots\\
			a_{kn}
		\end{matrix}\right)(  b_{k1}, b_{k2},\dots, b_{kn}) \geq 0,
	\end{equation}
	then $T_{\Phi}$ is hyponormal.
\end{corollary}
\begin{proof}
	Since $A,B$ are real matrices and (11) holds for $k = 1,2, \dots, n$, then for all $k = 1,2, \dots, n$,
	$$
	\operatorname{Re} \left( \left(\sum_{l=1}^n a_{k l} c_i^{(l)} \right) \left( \sum_{l=1}^n \bar{b}_{k l} \bar{c}_i^{(l)} \right) \right) = \operatorname{Re} \left( \left(\sum_{l=1}^n \bar{a}_{k l} c_i^{(l)} \right) \left( \sum_{l=1}^n b_{k l} \bar{c}_i^{(l)} \right) \right) \geq 0.
	$$
	A similar argument with Theorem 3.2's proof, then we complete the proof. 
\end{proof}
\begin{corollary}
	Given $p,s \in \mathbb{N}_+\cup {\left\{0\right\}}$, if $\Phi(z)=A \left| z \right|^{2p}+ B \left| z \right|^{2s}$, $ A = (a_{ij})_{n \times n} \in M_n $, $B = (b_{ij})_{n \times n} \in M_n $ satisfying $A^*=A$ and $B^*=B$, then $T_{\Phi}$ is normal hence hyponormal.
\end{corollary}
\begin{proof}
	Put $q = p$ and $t=s$ into (7), then $(2) \equiv 0$, hence $T_{\Phi}$ is normal. 
\end{proof} 
\begin{theorem}
	Given $ p,q,s,t,r \in \mathbb{N}_+\cup {\left\{0\right\}}$, $p-q=t-s > 0$, let $\Phi(z)=A z^p \bar{z}^q + B z^s \bar{z}^t$, $ A = (a_{ij})_{n \times n}, B = (b_{ij})_{n \times n} \in M_n $. If $T_{\Phi}$ is hyponormal then
	
	(i) If $t \geq p$, then for any $1 \leq r \leq n $, any sequence $1 \leq n_1 < n_2 \dots < n_r \leq n$ and any $c_1, c_2, \dots, c_r \in \mathbb{R}$,  
	\begin{align*}
		&\quad \sum_{k=1}^n \left| c_1 a_{k n_1} + c_2 a_{k n_2} + \dots + c_r a_{k n_r} \right|^2 \\
		&\geq \max \left\{ \frac{\Lambda_\alpha(t+i,s)}{\Lambda_\alpha(p+i,q)}, W_{\alpha}(p, q, t, s) \right\} \sum_{k=1}^n \left|c_1 b_{k n_1} + c_2 b_{k n_2} + \dots + c_r b_{k n_r} \right|^2.
	\end{align*}
	
	(ii) If $t<p$, then for any $1 \leq r \leq n $, any sequence $1 \leq n_1 < n_2 \dots < n_r \leq n$ and any $c_1, c_2, \dots, c_r \in \mathbb{R}$, 
	\begin{align*}
		&\quad \sum_{k=1}^n \left| c_1 a_{k n_1} + c_2 a_{k n_2} + \dots + c_r a_{k n_r} \right|^2 \\
		&\geq \max \left\{\frac{\Lambda_\alpha(t, s)}{\Lambda_\alpha(p, q)}, W_{\alpha}(p, q, t, s)\right\} \sum_{k=1}^n \left|c_1 b_{k n_1} + c_2 b_{k n_2} + \dots + c_r b_{k n_r} \right|^2.
	\end{align*}
	Where $W_{\alpha}(p, q, t, s)=\displaystyle\sup _{i \geq p-q} \frac{\Lambda_\alpha(t+i, s)-\Lambda_\alpha(s+i, t)}{\Lambda_\alpha(p+i, q)-\Lambda_\alpha(q+i, p)}$.
\end{theorem}
\begin{proof}
	By Lemma 3.1, then $T_{\Phi}$ is hyponormal, which is equivalent to
	\begin{equation}
		\begin{aligned}
			&\sum_{i=0}^{\infty} \Lambda_\alpha(p+i,q) \sum_{k=1}^n \left| \sum_{l=1}^n a_{k l} c_i^{(l)}\right|^2 - \sum_{i=p-q}^{\infty} \Lambda_\alpha(q+i,p) \sum_{k=1}^n \left| \sum_{l=1}^n \bar{a}_{k l} c_i^{(l)}\right|^2 \\
			& +\sum_{i=p-q}^{\infty} \Lambda_\alpha(s+i,t) \sum_{k=1}^n \left| \sum_{l=1}^n b_{k l} c_i^{(l)}\right|^2 - \sum_{i=0}^{\infty} \Lambda_\alpha(t+i,s) \sum_{k=1}^n \left| \sum_{l=1}^n \bar{b}_{k l} c_i^{(l)}\right|^2 \\
			& +2 \sum_{i=0}^{\infty} \sum_{k=1}^n \operatorname{Re}  \left( \left( \sum_{l=1}^n a_{k l} c_i^{(l)} \right) \left( \sum_{l=1}^n \bar{b}_{k l} \bar{c}_{2(p-q)+i}^{(l)} \right) \right) \frac{ \Lambda_\alpha(p+i)\Lambda_\alpha(p-q+t+i)}{\Lambda_\alpha(p-q+i)} \\
			& - 2 \sum_{i=0}^{\infty}  \sum_{k=1}^n \operatorname{Re} \left( \left( \sum_{l=1}^n \bar{a}_{k l} c_{2(p-q)+i}^{(l)} \right) \left( \sum_{l=1}^n b_{k l} \bar{c}_i^{(l)} \right) \right) \frac{\Lambda_\alpha(t+i) \Lambda_\alpha(t-s+p+i)}{\Lambda_\alpha(t-s+i)} \geq 0.
		\end{aligned}
	\end{equation}
	For $0 \leq i < p-q$, put $c_{i}^{(n_k)} = c_k$ $(k = 1,2,\dots,r)$, $c_{i}^{(l_1)} = 0$ $(l_1 \neq n_1, n_2 \dots n_r)$ and $c_{j}^{(l_2)} = 0$ $(j \neq i)$ for $ l_2 = 1,2, \dots , n$ into (12), then
	\begin{align*}
		&\quad \sum_{k=1}^n \left| c_1 a_{k n_1} + c_2 a_{k n_2} + \dots + c_r a_{k n_r} \right|^2 \\
		&\geq \frac{\Lambda_\alpha(t+i,s)}{\Lambda_\alpha(p+i,q)} \sum_{k=1}^n \left|c_1 b_{k n_1} + c_2 b_{k n_2} + \dots +c_r b_{k n_r} \right|^2.
	\end{align*}
	Therefore 
	\begin{equation}
		\begin{aligned}
			&\quad \sum_{k=1}^n \left| c_1 a_{k n_1} + c_2 a_{k n_2} + \dots + c_r a_{k n_r} \right|^2 \\
			&\geq \max_{0 \leq i<p-q} \frac{\Lambda_\alpha(t+i,s)}{\Lambda_\alpha(p+i,q)} \sum_{k=1}^n \left|c_1 b_{k n_1} + c_2 b_{k n_2} + \dots +c_r b_{k n_r} \right|^2.
		\end{aligned}
	\end{equation}
	
	To deal with (13), we divide it into two cases. Applying Lemma 2.2 (iii), then \\
	If $t \geq p$, it follows that $\displaystyle\frac{\Lambda_\alpha(t+i,s)}{\Lambda_\alpha(p+i,q)}$ is increasing in $i$; choose $i = p-q-1 = t-s-1$, then 
	\begin{equation}
		\begin{aligned}
			&\quad \sum_{k=1}^n \left|c_1 a_{k n_1} +c_2 a_{k n_2} + \dots +c_r a_{k n_r} \right|^2 \\
			&\geq \frac{\Lambda_\alpha(2t-s-1,s)}{\Lambda_\alpha(2p-q-1,q)} \sum_{k=1}^n \left|c_1 b_{k n_1} +c_2 b_{k n_2} + \dots + c_r b_{k n_r} \right|^2.
		\end{aligned}
	\end{equation}
	If $t<p$, it follows that $\displaystyle\frac{\Lambda_\alpha(t+i,s)}{\Lambda_\alpha(p+i,q)}$ is decreasing in $i$; choose $i= 0$, then
	\begin{equation}
		\begin{aligned}
			&\quad \sum_{k=1}^n \left| c_1 a_{k n_1} +c_2 a_{k n_2} + \dots + c_r a_{k n_r} \right|^2 \\
			&\geq \frac{\Lambda_\alpha(t,s)}{\Lambda_\alpha(p,q)} \sum_{k=1}^n \left|c_1 b_{k n_1} +c_2 b_{k n_2} + \dots +c_r b_{k n_r} \right|^2.
		\end{aligned}
	\end{equation}
	Similarly, for $i \geq p-q$, put $c_{i}^{(n_k)} = c_k$ $(k = 1,2,\dots,r)$, $c_{i}^{(l_1)} = 0$ $(l_1 \neq n_1, n_2 \dots n_r)$ and $c_{j}^{(l_2)} = 0$ $(j \neq i)$ for $ l_2 = 1,2, \dots , n$ into (12), then
	\begin{equation}
		\begin{aligned}
			&\quad \sum_{k=1}^n \left|c_1 a_{k n_1} + c_2 a_{k n_2} + \dots + c_r a_{k n_r} \right|^2 \\
			&\geq \sup _{i \geq p-q} \frac{\Lambda_\alpha(t+i, s)-\Lambda_\alpha(s+i, t)}{\Lambda_\alpha(p+i, q)-\Lambda_\alpha(q+i, p)} \sum_{k=1}^n \left|c_1 b_{k n_1} +c_2 b_{k n_2} + \dots +c_r b_{k n_r} \right|^2.
		\end{aligned}
	\end{equation}
	Combined with (14), (15), (16), we complete the proof.
\end{proof}
\begin{remark}
	The value of $\displaystyle \sup _{i \geq p-q} \frac{\Lambda_\alpha(t+i, s)-\Lambda_\alpha(s+i, t)}{\Lambda_\alpha(p+i, q)-\Lambda_\alpha(q+i, p)}$ can be calculated by taking the derivative when $p,q$ is known. Please refer to \cite{KL2023} for an example.
\end{remark}

\section{\hspace{-0.6cm}{\bf }~~Matrix-valued circulant symbols\label{s4}}
In this section, given $ p,q,s,t \in \mathbb{N}_+\cup {\left\{0\right\}}$, we consider two conditions for the hyponormality of $T_{\Phi}$ on $A_\alpha^2\left(\mathbb{C}^n\right)$, where $ \Phi(z)=A z^p \bar{z}^q + B z^s \bar{z}^t $, $ A,B $ are circulants. 

The following lemma on $A_\alpha^2\left(\mathbb{C}^n\right)$ from a direct calculation by Lemma 2.1.
\begin{lemma}
	Given $p,q \in \mathbb{N}_+\cup {\left\{0\right\}}$ and $a_k \in \mathbb{C} ~ (k = 1,2 \dots n) $, then 
	
	(i) $\left\|\operatorname{cir}\left[a_1, a_2, \cdots, a_{n}\right] \bar{z}^q K_p(z)\right\|^2 $
	$$
	\begin{aligned}
		~=\sum_{i=0}^{\infty} \Lambda_\alpha(p+i+q) \left[\sum_{k=1}^{n}\left|a_k\right|^2 \sum_{l=1}^n\left|c_{i}^{(l)}\right|^2+\sum_{k_1, k_2=1, k_1 \neq k_2}^{n} \bar{a}_{k_1} a_{k_2} \sum_{l_1, l_2=1, l_1 \neq l_2}^n \bar{c}_{i}^{\left(l_1\right)} c_{i}^{\left(l_2\right)}\right].		
	\end{aligned}
	$$
	
	(ii) If $p \geq q$, then
	\begin{align*}
		& \left\|P\left(\operatorname{cir}\left[a_1, a_2, \cdots, a_{n}\right] \bar{z}^q K_p(z)\right)\right\|^2 \\
		=&\sum_{i=0}^{\infty} \Lambda_\alpha(p+i,q) \left[\sum_{k=1}^{n}\left|a_k\right|^2 \sum_{l=1}^n\left|c_{i}^{(l)}\right|^2+\sum_{k_1, k_2=1, k_1 \neq k_2}^{n} \bar{a}_{k_1} a_{k_2} \sum_{l_1, l_2=1, l_1 \neq l_2}^n \bar{c}_{i}^{\left(l_1\right)} c_{i}^{\left(l_2\right)}\right].
	\end{align*}
	If $p < q$, then
	$$
	\begin{aligned}
		& \left\|P\left(\operatorname{cir}\left[a_1, a_2, \cdots, a_{n}\right] \bar{z}^q K_p(z)\right)\right\|^2 \\
		=&\sum_{i=q-p}^{\infty} \Lambda_\alpha(p+i,q) \left[\sum_{k=1}^{n}\left|a_k\right|^2 \sum_{l=1}^n\left|c_{i}^{(l)}\right|^2+\sum_{k_1, k_2=1, k_1 \neq k_2}^{n} \bar{a}_{k_1} a_{k_2} \sum_{l_1, l_2=1, l_1 \neq l_2}^n \bar{c}_{i}^{\left(l_1\right)} c_{i}^{\left(l_2\right)}\right].
	\end{aligned}
	$$
\end{lemma}

First, given $ p,q \in \mathbb{N}_+\cup {\left\{0\right\}}$, we consider the case when $\Phi(z)=\operatorname{cir}\left[a_1, a_2, \cdots, a_{n}\right] z^p \bar{z}^q $, $a_k \in \mathbb{C} ~ (k = 1,2 \dots n) $.
\begin{theorem}
	Given $ p,q \in \mathbb{N}_+\cup {\left\{0\right\}}$, let $\Phi(z)=\operatorname{cir}\left[a_1, a_2, \cdots, a_{n}\right] z^p \bar{z}^q $, $\operatorname{cir}\left[a_1, a_2, \cdots, a_{n}\right] \neq \mathbb{O}$, $a_k \in \mathbb{C} ~ (k = 1,2 \dots n) $. If
	$$
	\sum_{k=1}^{n}\left|a_k\right|^2 \geq (n-1) \left| \sum_{k_1, k_2=1, k_1 \neq k_2}^{n} \bar{a}_{k_1} a_{k_2}\right|,
	$$
	then $T_{\Phi}$ is hyponormal if and only if $p \geq q$.
	
\end{theorem}
\begin{proof}
	If $p \geq q$, by Lemma 4.1 then $T_{\Phi}$ is hyponormal, which is equivalent to
	\begin{equation}
		\begin{aligned}
			&\sum_{i=0}^{\infty} \Lambda_\alpha(p+i,q) \left[\sum_{k=1}^{n}\left|a_k\right|^2 \sum_{l=1}^n\left|c_{i}^{(l)}\right|^2+\sum_{k_1, k_2=1, k_1 \neq k_2}^{n} \bar{a}_{k_1} a_{k_2} \sum_{l_1, l_2=1, l_1 \neq l_2}^n \bar{c}_{i}^{\left(l_1\right)} c_{i}^{\left(l_2\right)}\right] \\
			& -\sum_{i=p-q}^{\infty} \Lambda_\alpha(q+i,p) \left[\sum_{k=1}^{n}\left|a_k\right|^2 \sum_{l=1}^n\left|c_{i}^{(l)}\right|^2+\sum_{k_1, k_2=1, k_1 \neq k_2}^{n} \bar{a}_{k_1} a_{k_2} \sum_{l_1, l_2=1, l_1 \neq l_2}^n \bar{c}_{i}^{\left(l_1\right)} c_{i}^{\left(l_2\right)}\right] \geq 0 .
		\end{aligned} 
	\end{equation}
	Since
	\begin{equation}
		\sum_{l=1}^n\left|c^{(l)}\right|^2 \geq \frac{1}{n-1}\left| \sum_{l_1, l_2=1, l_1 \neq l_2}^n\bar{c}^{\left(l_1\right)} c^{\left(l_2\right)}\right|,
	\end{equation}
	then
	$$
	\sum_{k=1}^{n}\left|a_k\right|^2 \sum_{l=1}^n\left|c_{i}^{(l)}\right|^2+\sum_{k_1, k_2=1, k_1 \neq k_2}^{n} \bar{a}_{k_1} a_{k_2} \sum_{l_1, l_2=1, l_1 \neq l_2}^n \bar{c}_{i}^{\left(l_1\right)} c_{i}^{\left(l_2\right)} \geq 0.
	$$
	Applying Lemma 2.2 (i), then (17) holds, hence $T_{\Phi}$ is hyponormal. 
	
	Similarly, if $p<q$, then $T_{\Phi^*}$ is hyponormal, hence $T_{\Phi}$ is never hyponormal. 
\end{proof}	

\begin{corollary}
	Given $ p,q \in \mathbb{N}_+\cup {\left\{0\right\}}$, $p>q$, let $\Phi(z)=\operatorname{cir} [a_1, a_2, \cdots, a_{n}] z^p \bar{z}^q + \operatorname{cir}[b_1, b_2, \cdots, b_{n}] z^q \bar{z}^p$, $a_k, b_k \in \mathbb{C} ~ (k = 1,2 \dots n) $. Then
	
	(i) If
	$$
	\sum_{k_1, k_2=1, k_1 \neq k_2}^{n} (\bar{a}_{k_1} a_{k_2}-\bar{b}_{k_1} b_{k_2}) \leq 0,
	$$
	then $T_{\Phi}$ is hyponormal if and only if 
	\begin{equation}
		\sum_{k=1}^{n}(\left|a_k\right|^2-\left|b_k\right|^2) + (n-1)  \sum_{k_1, k_2=1, k_1 \neq k_2}^{n} (\bar{a}_{k_1} a_{k_2}-\bar{b}_{k_1} b_{k_2}) \geq 0.
	\end{equation}
	
	(ii) If
	$$
	\sum_{k_1, k_2=1, k_1 \neq k_2}^{n} (\bar{a}_{k_1} a_{k_2}-\bar{b}_{k_1} b_{k_2}) > 0,
	$$
	then $T_{\Phi}$ is hyponormal if and only if 
	\begin{equation}
		\sum_{k=1}^{n}(\left|a_k\right|^2-\left|b_k\right|^2) - \sum_{k_1, k_2=1, k_1 \neq k_2}^{n} (\bar{a}_{k_1} a_{k_2}-\bar{b}_{k_1} b_{k_2}) \geq 0.
	\end{equation}
\end{corollary}
\begin{proof}
	By Lemma 4.1, then $T_{\Phi}$ is hyponormal, which is equivalent to
	\begin{equation}
		\begin{aligned}
			& \sum_{i=0}^{\infty} \Lambda_\alpha(p+i,q) \left[\sum_{k=1}^{n} \left( \left|a_k\right|^2 - \left|b_k\right|^2 \right) \sum_{l=1}^n\left|c_{i}^{(l)}\right|^2 \right. \\
			&\quad\left. +\sum_{k_1, k_2=1, k_1 \neq k_2}^{n} \left( \bar{a}_{k_1} a_{k_2} - \bar{b}_{k_1} b_{k_2} \right) \sum_{l_1, l_2=1, l_1 \neq l_2}^n \bar{c}_{i}^{\left(l_1\right)} c_{i}^{\left(l_2\right)}\right]  \\
			& -\sum_{i=p-q}^{\infty} \Lambda_\alpha(q+i,p) \left[\sum_{k=1}^{n} \left( \left|a_k\right|^2 - \left|b_k\right|^2 \right) \sum_{l=1}^n\left|c_{i}^{(l)}\right|^2 \right.\\
			&\quad\left.+\sum_{k_1, k_2=1, k_1 \neq k_2}^{n} \left( \bar{a}_{k_1} a_{k_2} - \bar{b}_{k_1} b_{k_2} \right) \sum_{l_1, l_2=1, l_1 \neq l_2}^n \bar{c}_{i}^{\left(l_1\right)} c_{i}^{\left(l_2\right)}\right] \geq 0 .
		\end{aligned}
	\end{equation}
	If $T_{\Phi}$ is hyponormal, put $ c_{0}^{(1)} = 1$, $ c_{0}^{(l_1)} = 0$ for $ l_1 =2, \dots , n$ and $ c_{i}^{(l_2)} = 0$ $(i \geq 1)$ for $ l_2 = 1,2, \dots , n$ into (21) then
	\begin{equation}
		\sum_{k=1}^{n} \left( \left|a_k\right|^2 - \left|b_k\right|^2 \right) \geq 0.
	\end{equation}
	By Lemma 2.2 (i), it follows that $T_{\Phi}$ is hyponormal if for all $i = 0,1,2, \dots$.
	\begin{equation}
		\begin{aligned}
			\sum_{k=1}^{n} \left( \left|a_k\right|^2 - \left|b_k\right|^2 \right) \sum_{l=1}^n\left|c_{i}^{(l)}\right|^2+\sum_{k_1, k_2=1, k_1 \neq k_2}^{n} \left( \bar{a}_{k_1} a_{k_2} - \bar{b}_{k_1} b_{k_2} \right) \sum_{l_1, l_2=1, l_1 \neq l_2}^n \bar{c}_{i}^{\left(l_1\right)} c_{i}^{\left(l_2\right)} \geq 0.
		\end{aligned}
	\end{equation}
	If
	$$
	\sum_{k_1, k_2=1, k_1 \neq k_2}^{n} (\bar{a}_{k_1} a_{k_2}-\bar{b}_{k_1} b_{k_2}) \leq 0,
	$$
	it follows (18), (22), (23) that $T_{\Phi}$ is hyponormal if (19) holds. If $T_{\Phi}$ is hyponormal, put $ c_{0}^{(l)} = 1$ and $ c_{i}^{(l)} = 0$ $(i \geq 1)$ for $ l = 1,2, \dots , n$ into (21), then (19) holds, which proves (i). If
	$$
	\sum_{k_1, k_2=1, k_1 \neq k_2}^{n} (\bar{a}_{k_1} a_{k_2}-\bar{b}_{k_1} b_{k_2}) > 0,
	$$
	from (22), (23), $T_{\Phi}$ is hyponormal if for all $i = 0,1,2 , \dots$.
	$$
	\frac{\displaystyle \sum_{k=1}^{n}(\left|a_k\right|^2-\left|b_k\right|^2)}{ \displaystyle \sum_{k_1, k_2=1, k_1 \neq k_2}^{n} (\bar{a}_{k_1} a_{k_2}-\bar{b}_{k_1} b_{k_2})}  \geq  - \frac{ \displaystyle \sum_{l_1, l_2=1, l_1 \neq l_2}^n \bar{c}_{i}^{\left(l_1\right)} c_{i}^{\left(l_2\right)}}{ \displaystyle \sum_{l=1}^n\left|c_{i}^{(l)}\right|^2} .
	$$
	Since 
	$$
	-\frac{ \displaystyle \sum_{l_1, l_2=1, l_1 \neq l_2}^n \bar{c}_{i}^{\left(l_1\right)} c_{i}^{\left(l_2\right)}}{ \displaystyle \sum_{l=1}^n\left|c_{i}^{(l)}\right|^2} = \frac{ \displaystyle \sum_{l=1}^n\left|c_{i}^{(l)}\right|^2 - \sum_{l=1}^n c_{i}^{(l)} \sum_{l=1}^n \bar{c}_{i}^{(l)}}{ \displaystyle \sum_{l=1}^n\left|c_{i}^{(l)}\right|^2} = 1 - \frac{ \displaystyle \sum_{l=1}^n c_{i}^{(l)} \sum_{l=1}^n \bar{c}_{i}^{(l)}}{ \displaystyle \sum_{l=1}^n\left|c_{i}^{(l)}\right|^2} \leq 1.
	$$
	The above inequality takes the equal sign if $ \sum_{l=1}^n c_{i}^{(l)} = 0$ ($c_{i}^{(l)}$'s are not all 0's), then $T_{\Phi}$ is hyponormal if (20) holds. If $T_{\Phi}$ is hyponormal, put $ c_{i}^{(l)} = 0$ $ (i \geq 1)$ for $ l = 1,2, \dots , n$ into (21) and select not all 0's $ c_{0}^{(l)}$ for $ l = 1,2, \dots , n$ satisfying $\sum_{l=1}^n c_{l}^{(l)} = 0$ into (21), then (20) holds, which proves (ii).
\end{proof}

Next, given $ p,q,s,t \in \mathbb{N}_+\cup {\left\{0\right\}}$, we consider the case when $\Phi = \operatorname{cir}\left[a_1, a_2, \cdots, a_{n}\right] z^p \bar{z}^q + \operatorname{cir}\left[b_1, b_2, \cdots, b_{n}\right] z^s \bar{z}^t$, $a_k, b_k \in \mathbb{C} ~ (k = 1,2 \dots n) $.
\begin{theorem}
	Given $ p,q,s,t \in \mathbb{N}_+\cup {\left\{0\right\}}$ $p, q, s, t \in \mathbb{N}_+\cup {\left\{0\right\}}$, $p-q=s-t > 0$, let $\Phi(z)=\operatorname{cir}\left[a_1, a_2, \cdots, a_{n}\right] z^p \bar{z}^q + \operatorname{cir}\left[b_1, b_2, \cdots, b_{n}\right] z^s \bar{z}^t$, $a_k, b_k \in \mathbb{C} ~ (k = 1,2 \dots n) $. If
	$$
	\sum_{k=1}^{n}\left|a_k\right|^2 \geq (n-1) \left| \sum_{k_1, k_2=1, k_1 \neq k_2}^{n}\bar{a}_{k_1} a_{k_2}\right|,
	$$
	and
	$$
	\sum_{k=1}^{n}\left|b_k\right|^2 \geq (n-1) \left| \sum_{k_1, k_2=1, k_1 \neq k_2}^{n}\bar{b}_{k_1} b_{k_2}\right|,
	$$
	and
	$$
	\sum_{k=1}^{n} \operatorname{Re} (\bar{a}_k b_k) \geq (n-1) \left| \sum_{k_1, k_2=1, k_1 \neq k_2}^{n}  \operatorname{Re} ({b}_{k_1} \bar{a}_{k_2}) \right|,
	$$
	then $T_{\Phi}$ is hyponormal.
\end{theorem} 
\begin{proof}
	By (18) then
	\begin{equation}
		\sum_{k=1}^{n}\left|a_k\right|^2 \sum_{l=1}^n\left|c_{i}^{(l)}\right|^2+\sum_{k_1, k_2=1, k_1 \neq k_2}^{n} \bar{a}_{k_1} a_{k_2} \sum_{l_1, l_2=1, l_1 \neq l_2}^n \bar{c}_{i}^{\left(l_1\right)} c_{i}^{\left(l_2\right)} \geq 0,
	\end{equation}
	and
	\begin{equation}
		\sum_{k=1}^{n}\left|b_k\right|^2 \sum_{l=1}^n\left|c_{i}^{(l)}\right|^2+\sum_{k_1, k_2=1, k_1 \neq k_2}^{n} \bar{b}_{k_1} b_{k_2} \sum_{l_1, l_2=1, l_1 \neq l_2}^n \bar{c}_{i}^{\left(l_1\right)} c_{i}^{\left(l_2\right)} \geq 0,
	\end{equation}
	and
	\begin{equation}
		\begin{aligned}
			&\operatorname{Re} \left( \sum_{k=1}^{n} \bar{a}_k b_k \sum_{l=1}^n\left|c_{i}^{(l)}\right|^2+\sum_{k_1, k_2=1, k_1 \neq k_2}^{n} {b}_{k_1} \bar{a}_{k_2} \sum_{l_1, l_2=1, l_1 \neq l_2}^n \bar{c}_{i}^{\left(l_1\right)} c_{i}^{\left(l_2\right)} \right) \\
			=&\sum_{k=1}^{n} \operatorname{Re} (\bar{a}_k b_k) \sum_{l=1}^n\left|c_{i}^{(l)}\right|^2+\sum_{k_1, k_2=1, k_1 \neq k_2}^{n}  \operatorname{Re} ({b}_{k_1} \bar{a}_{k_2}) \sum_{l_1, l_2=1, l_1 \neq l_2}^n \bar{c}_{i}^{\left(l_1\right)} c_{i}^{\left(l_2\right)} \geq 0.
		\end{aligned}
	\end{equation}
	By Lemma 4.1, then $T_{\Phi}$ is hyponormal, which is equivalent to
	\begin{align}
		& =\sum_{i=0}^{\infty} \Lambda_\alpha(p+i,q) \left[\sum_{k=1}^{n}\left|a_k\right|^2 \sum_{l=1}^n\left|c_{i}^{(l)}\right|^2+\sum_{k_1, k_2=1, k_1 \neq k_2}^{n} \bar{a}_{k_1} a_{k_2} \sum_{l_1, l_2=1, l_1 \neq l_2}^n \bar{c}_{i}^{\left(l_1\right)} c_{i}^{\left(l_2\right)}\right] \nonumber\\
		&\quad -\sum_{i=p-q}^{\infty} \Lambda_\alpha(q+i,p) \left[\sum_{k=1}^{n}\left|a_k\right|^2 \sum_{l=1}^n\left|c_{i}^{(l)}\right|^2+\sum_{k_1, k_2=1, k_1 \neq k_2}^{n} \bar{a}_{k_1} a_{k_2} \sum_{l_1, l_2=1, l_1 \neq l_2}^n \bar{c}_{i}^{\left(l_1\right)} c_{i}^{\left(l_2\right)}\right] \nonumber\\
		&\quad+\sum_{i=0}^{\infty} \Lambda_\alpha(s+i,t) \left[\sum_{k=1}^{n}\left|b_k\right|^2 \sum_{l=1}^n\left|c_{i}^{(l)}\right|^2+\sum_{k_1, k_2=1, k_1 \neq k_2}^{n} \bar{b}_{k_1} b_{k_2} \sum_{l_1, l_2=1, l_1 \neq l_2}^n \bar{c}_{i}^{\left(l_1\right)} c_{i}^{\left(l_2\right)}\right] \nonumber\\
		& \quad-\sum_{i=t-s}^{\infty} \Lambda_\alpha(t+i,s) \left[\sum_{k=1}^{n}\left|b_k\right|^2 \sum_{l=1}^n\left|c_{i}^{(l)}\right|^2+\sum_{k_1, k_2=1, k_1 \neq k_2}^{n} \bar{b}_{k_1} b_{k_2} \sum_{l_1, l_2=1, l_1 \neq l_2}^n \bar{c}_{i}^{\left(l_1\right)} c_{i}^{\left(l_2\right)}\right] \nonumber\\
		&\quad+2\operatorname{Re}\left\{ \sum_{i=0}^{\infty} \frac{\Lambda_\alpha(p+i) \Lambda_\alpha(s+i)}{\Lambda_\alpha(p-q+i)} \right. \nonumber\\
		&\left.\quad\quad\times \left[\sum_{k=1}^{n} a_k \bar{b}_k \sum_{l=1}^n\left|c_{i}^{(l)}\right|^2+\sum_{k_1, k_2=1, k_1 \neq k_2}^{n} \bar{b}_{k_1} a_{k_2} \sum_{l_1, l_2=1, l_1 \neq l_2}^n \bar{c}_{i}^{\left(l_1\right)} c_{i}^{\left(l_2\right)}\right] \right\} \nonumber\\
		&\quad-2 \operatorname{Re} \left\{\sum_{i=p-q}^{\infty} \frac{\Lambda_\alpha(q+i) \Lambda_\alpha(t+i)}{\Lambda_\alpha(q-p+i)} \right. \nonumber\\
		&\left.\quad\quad\times \left[\sum_{k=1}^{n} \bar{a}_k b_k \sum_{l=1}^n\left|c_{i}^{(l)}\right|^2+\sum_{k_1, k_2=1, k_1 \neq k_2}^{n} {b}_{k_1} \bar{a}_{k_2} \sum_{l_1, l_2=1, l_1 \neq l_2}^n \bar{c}_{i}^{\left(l_1\right)} c_{i}^{\left(l_2\right)}\right] \right\} \geq 0.
	\end{align}
	Combined with (24), (25), (26) and Lemma 2.2 (i) (ii), it follows that $T_{\Phi}$ is hyponormal. 
\end{proof}	
\begin{corollary}
	Given $p,s \in \mathbb{N}_+\cup {\left\{0\right\}}$, if $\Phi(z)=\operatorname{cir}\left[a_1, a_2, \cdots, a_{n}\right] \left| z \right|^{2p}+ \operatorname{cir}\left[b_1, b_2, \cdots, b_{n}\right] \left| z \right|^{2s}$, $a_k, b_k \in \mathbb{C} ~ (k = 1,2 \dots n) $, then $T_{\Phi}$ is normal hence hyponormal.
\end{corollary}
\begin{proof}
	Put $q = p$ and $t=s$ into (27), then $(2) \equiv 0$, hence $T_{\Phi}$ is normal.
\end{proof}

Since the circulant is a particular instance of  an n-order complex square matrix, applying Theorem 3.3, we immediately obtain the necessary condition for $T_{\Phi}$ to be hyponormal, where $\Phi = \operatorname{cir}\left[a_1, a_2, \cdots, a_{n}\right] z^p \bar{z}^q + \operatorname{cir}\left[b_1, b_2, \cdots, b_{n}\right] z^s \bar{z}^t$, $a_k, b_k \in \mathbb{C} ~ (k = 1,2 \dots n) $. Particularly for the circulant, we have the following interesting example:
\begin{example}
	Given $ p,q,s,t \in \mathbb{N}_+\cup {\left\{0\right\}}$, $p-q=t-s > 0$, let $\Phi(z)=\operatorname{cir}\left[a_1, a_2, \cdots, a_{n}\right] z^p \bar{z}^q + \operatorname{cir}\left[b_1, b_2, \cdots, b_{n}\right] z^s \bar{z}^t$, $a_k, b_k \in \mathbb{C} ~ (k = 1,2 \dots n) $. If $T_{\Phi}$ is hyponormal then
	
	(i) If $t \geq p$, then 
	\begin{align*}
		&\quad \sum_{k=1}^{n}\left|a_k\right|^2 - \sum_{k_1, k_2=1, k_1 \neq k_2}^{n} \bar{a}_{k_1} a_{k_2} \\
		&\geq \max \left\{\frac{\Lambda_\alpha(2 t-s-1, s)}{\Lambda_\alpha(2 p-q-1, q)}, W_{\alpha}(p, q, t, s)\right\} \left[\sum_{k=1}^{n}\left|b_k\right|^2 - \sum_{k_1, k_2=1, k_1 \neq k_2}^{n} \bar{b}_{k_1} b_{k_2} \right]. 
	\end{align*}
	
	(ii) If $t<p$, then
	\begin{align*}
		&\quad \sum_{k=1}^{n}\left|a_k\right|^2 - \sum_{k_1, k_2=1, k_1 \neq k_2}^{n} \bar{a}_{k_1} a_{k_2} \\
		&\geq \max \left\{\frac{\Lambda_\alpha(t, s)}{\Lambda_\alpha(p, q)}, W_{\alpha}(p, q, t, s)\right\} \left[\sum_{k=1}^{n}\left|b_k\right|^2 - \sum_{k_1, k_2=1, k_1 \neq k_2}^{n} \bar{b}_{k_1} b_{k_2} \right]. 
	\end{align*}
	Where $W_{\alpha}(p, q, t, s)=\displaystyle \sup _{i \geq p-q} \frac{\Lambda_\alpha(t+i, s)-\Lambda_\alpha(s+i, t)}{\Lambda_\alpha(p+i, q)-\Lambda_\alpha(q+i, p)}$.
\end{example}
\begin{proof}
	Immediate from Theorem 3.3 $(r=n, c_1=c_2=\dots=c_{n-1} = 1, c_n = 1-n)$. 
\end{proof}


\hspace*{-0.6cm}\textbf{\bf Competing interests}\\
The authors declare that they have no competing interests.\\

\hspace*{-0.6cm}\textbf{\bf Funding}\\
The research was supported by Natural Science Foundation of China (Grant Nos. 12061069).\\

\hspace*{-0.6cm}\textbf{\bf Authors contributions}\\
All authors contributed equality and significantly in writing this paper. All authors read and approved the final manuscript.\\

\hspace*{-0.6cm}\textbf{\bf Acknowledgments}\\
The authors would like to express their thanks to the referees for valuable advice regarding previous version of this paper.\\

\hspace*{-0.6cm}\textbf{\bf Authors details}\\
Guangyang Fu and Jiang Zhou*, 107552300585@stu.xju.edu.cn and zhoujiang@xju.edu.cn, College of Mathematics and System Science, Xinjiang University, Urumqi, 830046, P.R China.

\bigskip
\noindent Guangyang Fu and Jiang Zhou\\
\medskip
\noindent
College of Mathematics and System Sciences\\
Xinjiang University\\
Urumqi 830046\\
\smallskip
\noindent{E-mail }:\\
\texttt{107552300585@stu.xju.edu.cn} (Guangyang Fu)\\
\texttt{zhoujiang@xju.edu.cn} (Jiang Zhou)
\bigskip \medskip
\end{document}